\documentclass[twoside,12pt]{article}
\usepackage{amsmath}
\usepackage[dvipsnames,usenames]{color}
\usepackage{amsfonts}
\usepackage{mathrsfs}
\usepackage{amssymb}
\usepackage{amsthm}
\def\<{\langle}
\def\>{\rangle}
\numberwithin{equation}{section}
\def\<{\langle}
\def\>{\rangle}

\def\-{\overline}

\def\-{\overline}

\usepackage{fancyhdr}
\newtheorem{theorem}{Theorem}[section]
\newtheorem{lemma}[theorem]{Lemma}
\newtheorem{Corollary}[theorem]{Corollary}

\newtheorem{Definition}[theorem]{Definition}

\newtheorem{example}[theorem]{Example}
\newtheorem{remark}[theorem]{Remark}

\begin{document}
\title{\bf On the regularity of CR mappings between CR manifolds of hypersurface type}

\author{S. Berhanu \thanks{Work supported in part by NSF DMS 1300026}\\Department of Mathematics\\Temple University\\Philadelphia, PA 19122 \and Ming Xiao\\Department of Mathematics\\Rutgers University\\New Brunswick, NJ 08854}

  \maketitle

\begin{abstract}
We prove smooth and analytic versions of the classical Schwarz reflection principle for transversal CR mappings between two Levi-nondegenerate CR manifolds of hypersurface type.
\end{abstract}


\section{Introduction}
This article studies the smoothness and analyticity of CR mappings between two CR manifolds $M$ (the source) and $M'$ (the target) where $M$ is an abstract CR manifold of hypersurface type of CR dimension $n$ and $M'\subset \mathbb C^{N+1}$ ($N>n \geq 1$) is a hypersurface. It will be assumed that both $M$ and $M'$ are Levi nondegenerate. When the target manifold $M'$ is strongly pseudoconvex, smoothness results for CR mappings were proved in our work [BX] under weaker assumptions on the abstract CR manifold $M$. For example, one of the results in [BX] showed that when $M'\subset \mathbb C^{n+k}$ is strongly pseudoconvex and smooth, and $M$ is a smooth abstract CR manifold of CR dimension $n$ such that its Levi form at each characteristic covector has a nonzero eigenvalue, then any $C^k$ CR mapping whose differential is injective on the CR bundle of $M$ is smooth on a dense open set. Here we prove an analogous result when $M'$ is assumed to be Levi nondegenerate, but not strongly pseudoconvex. We show by means of an example that in this case, one has to impose a restriction on the signature of $M'$.

The main results of the paper may be viewed as smooth and analytic versions of the Schwarz reflection principle for CR maps. Among the numerous works on various versions of this principle, we mention [Fe], [Le], [Pi], [BJT], [Be], [Fr1], [Fr2], [BN], [CKS], [CGS], [CS], [DW], [EH], [EL], [Hu1],[Hu2], [KP], [La1], [La2], [La3], [M], [NWY], [Tu], and [W]. The book [BER] contains an account and many more references when the manifolds are real analytic or algebraic.

In Section 2 we state our main result and present two examples that illustrate why the assumptions in our main result can not be relaxed. In the same section, in order to substantiate our examples, we construct an example of a $C^k$ CR function (for any positive integer $k$) on a smooth strongly pseudoconvex manifold $M$ that is not $C^{\infty}$ on any nonempty open set in $M$. In Section 3 we present the proof of our main result.

\bigskip

\bigskip

{\bf Acknowledgement:} The authors thank Professor Xiaojun Huang for a helpful discussion on this work.

\section{Main result and preliminaries}

Let $M$ and $M'$ be  CR manifolds with CR bundles $\mathcal V$ and $\mathcal V'$ respectively. A differentiable mapping $F:M\to M'$ is called a CR mapping if $dF(\mathcal V)\subset \mathcal V'$. When $M' \subset \mathbb C^{N'}$, this is equivalent to saying that the components of $F=(F_1,\dots,F_{N'})$ are CR functions. The mapping $F$ is called CR transversal at $p\in M$ if $dF(\mathbb CT_pM)$ is not contained in $\mathcal V'_{F(p)}+\overline{\mathcal V'_{F(p)}}$.

The main result of this article is as follows:

\begin{theorem} \label{thm1}
Let $M$ be a smooth abstract CR manifold of hypersurface type of CR dimension $n$ and $M' \subset \mathbb{C}^{N+1},(n \geq 1, n < N \leq 2n)$ be a smooth real hypersurface. Assume that $M$ and $M'$ are Levi-nondegenerate and $M'$ has signature $(l,N-l),~l>0$ the number of positive eigenvalues of the Levi form. Let $F=(F_{1},\cdots,F_{N+1}): M \rightarrow M'$ be a CR-transversal CR mapping of class $C^{N-n+1}.$ Assume that $l \leq n$ and $N-l \leq n.$ Then $F$ is smooth on a dense open subset of $M.$
\end{theorem}

We remark that if $l>n$ or $N-l>n$, Example $2.4$ will show that the Theorem will not hold. This explains the assumption $N\leq 2n$ in Theorem $2.1$. Note that the case $l=0$ (and therefore also $l=N$) was treated in $[BX]$, and therefore, we may always assume that $0<l<N$. Since CR functions are $C^{\infty}$ whenever the Levi form has a positive and a negative eigenvalue, we may also assume that $M$ is strongly pseudoconvex. Our methods also lead to the following analyticity result:

\begin{theorem}\label{thm2.1}
Let $M\subset \mathbb C^{n+1}$  and $M' \subset \mathbb{C}^{N+1},(n \geq 1, n < N \leq 2n)$ be real analytic hypersurfaces. Assume that $M$ and $M'$ are Levi-nondegenerate and $M'$ has signature $(l,N-l),~l>0$ the number of positive eigenvalues of the Levi form. Let $F=(F_{1},\cdots,F_{N+1}): M \rightarrow M'$ be a CR-transversal CR mapping of class $C^{N-n+1}.$ Assume that $l \leq n$ and $N-l \leq n.$ Then $F$ is real analytic on a dense open subset of $M.$
\end{theorem}
It is well known that in Theorem $2.2$, if $M_1\subset M$ denotes the dense subset where $F$ is real analytic, then $F$ extends as a holomorphic map in a neighborhood of each point of $M_1$. As a consequence of Theorem 2.1, Theorem 2.2 and the main result in [BX], we have the following:

\begin{Corollary}\label{coro23}
Let $M\subset \mathbb{C}^{n}$  and  $M' \subset \mathbb{C}^{n+1} (n \geq 2)$ be real analytic (resp. smooth) hypersurfaces. Assume that $M$ and $M'$ are  Levi-nondegenerate and $F: M \rightarrow M'$ is a CR-transversal CR mapping of class $C^2.$ Then $F$ is real analytic (resp. smooth) on a dense open subset of $M.$
\end{Corollary}

When $F$ is assumed to be $C^{\infty}$, Corollary \ref{coro23} in the real analytic case was proved in [EL]. Corollary \ref{coro23} implies that a result on finite jet determination proved in [EL] (see Corollary 1.3 in [EL]) holds under a milder smoothness assumption:

\begin{Corollary}
Let $M\subset \mathbb{C}^{n}$ and $M'\subset \mathbb{C}^{n+1} (n\geq 2)$ be smooth connected hypersurfaces which are Levi-nondegenerate, and $f:M \rightarrow M'$ and $g:M\rightarrow M'$  transversal CR mappings of class $C^2$. If for any $p$ in some dense open subset of $M$, the jets at $p$  of $f$ and $g$ satisfy $j_p^4f=j_p^4g$, then $f=g$.
\end{Corollary}
The following examples show that in Theorems $2.1$ and $2.2$, neither the hypothesis on the signature of $M'$ nor the transversality assumption on $F$ can be dropped.

\begin{example}\label{eg1}
Let $M \subset \mathbb{C}^{n+1} (n \geq 1)$ be the hypersurface given by \newline $\{(z_{1},\dots,z_{n},w) \in \mathbb{C}^{n+1} :\mathrm{Im}\,w=\sum_{i=1}^{n}|z_{i}|^{2}\}.$ Let $M' \subset \mathbb{C}^{N+1}\, (N \geq n+2)$ be defined as $\{(z_{1},\cdots,z_{N},w) \in \mathbb{C}^{N+1}: \mathrm{Im} \,w= \sum_{i=1}^{n+1}|z_{i}|^2 +\sum_{j=n+2}^{N-1}\epsilon_j|z_{j}|^2 - |z_{N}|^2\}$, where each $\epsilon_j\in\{1,-1\}$.  Let $f$
be a $C^{N-n+1}$ CR function on $M$ which not smooth on any nonempty open subset of $M$ (see Theorem $2.7$ below for an example of such).
Then
$F(z_1,\dots,z_n,w)=(z_{1},\dots,z_{n},f(z_1,\dots,z_n,w),0,\dots 0,f(z_1,\dots,z_n,w),w)$ is a CR-transversal map of class $C^{N-n+1}$ from $M$ to $M'.$  Clearly $F$ is not smooth on any nonempty open subset of $M$ and, hence, since we may assume in Theorem \ref{thm1} that $M'$ is not strongly pseudoconvex and that therefore when $l>n$, $N\geq n+2$, Theorem $2.1$ does not hold. Likewise, the theorem does not hold when $N-l>n$. It follows that for the theorem to hold, we need to assume that $l\leq n,\,N-l\leq n$ and hence $N\leq 2n$.
\end{example}

\begin{example}\label{eg2}
Let $M \subset \mathbb{C}^{n}\,(n \geq 2)$ be given by $\{(z_{1},\cdots,z_{n-1},w) \in \mathbb{C}^{n}:\mathrm{Im}\,w=\sum_{i=1}^{n-1}|z_{i}|^2\}$ and define $M' \subset \mathbb{C}^{n+1}$ by  $M'=\{(z_{1},\cdots,z_{n},w) \in \mathbb{C}^{n+1}:\mathrm{Im}\,w= \sum_{i=1}^{n-1}|z_{i}|^{2}-|z_{n}|^2\}.$ Then $F=(0,\cdots,0,f,f,0)$ is a $C^2$ CR map from $M$ to $M',$ where $f$ is a $C^2$ CR function on $M$ which is not smooth on any nonempty open subset of $M.$ Note that $F$ is not transversal at any point on $M,$ and is not smooth on any nonempty open subset of $M.$
\end{example}

In order to make the preceding two examples meaningful, we will next show the existence of a $C^k$ CR function on a strongly pseudoconvex hypersurface which is not smooth on any nonempty open subset.

\begin{theorem}\label{thm24}
Let $D\subset \mathbb C^n$ be a bounded domain with a smooth boundary $M$ which is strongly pseudoconvex. Let $k\geq 1$ be a positive integer. Then there exists a CR function $f$ on $M$ of class $C^k$ which is not $C^{\infty}$ on any nonempty open subset of $M$.
\end{theorem}
\begin{proof}
First fix $p\in M$ and let $g\in C^{\infty}(\overline D)$ that is holomorphic on $D$ and peaks at $p$, say, $|g(z)|<g(p)=1$ for $z\in \overline D \setminus {p}$. By Hopf's Lemma, the normal derivative of $g$ at $p$ is nonzero and hence there is a smooth vector field $X$ tangent to $M$ near $p$ such that $Xg(p)\neq 0$. It follows that for any positive integer $m$, with a choice of a branch of logarithm, the function $g_m(z)=(1-g(z))^{m+\frac{1}{2}}$ is a CR function of class $C^m$ on $M$ but which is not of class $C^{m+1}$ at $p$. Let $\{p_{i}\}_{i=0}^{\infty} \subset M$ be a  dense subset of $M.$ We choose a sequence of $C^k$ CR functions $\{f_{i}\}_{i=0}^{\infty}$ on $M$ with the following properties: For each $i \geq 0, f_{i} \in C^{k+i}(M) \cap C^{\infty}(M \setminus \{p_{i}\}),$ and $f_{i}$ is not $C^{k+i+1}$ at $p_{i}.$ Then there exists a sequence of positive numbers $\{b_{i}\}_{i=0}^{\infty}$ such that, for any sequence of complex numbers $\{c_{i}\}_{i=0}^{\infty}$ with $|c_{i}| \leq b_{i}, i \geq 0, \sum_{i=0}^{\infty}c_{i}f_{i}$ converges uniformly to a $C^k$ CR function on $M.$

We fix a local chart $(U_{i},x)$ for each $p_{i},i\geq 0$ on $M,$ where $U_{i}$ is a neighborhood of $p_{i}.$ Choose $\Omega_{i} \subset\subset U_{i}, i\geq 0$ to be a  sufficiently small neighborhood of $p_{i}$ with the following properties:

\bigskip

(1).  For each $i \geq 1, p_{0},\cdots,p_{i-1} \not\in \Omega_{i}.$

(2). There exists a sequence of positive numbers $\{M_{i}^{j}\}_{i> j}$ such that for any $j \geq 0,$
$|D^{\alpha}f_{i}(x)| \leq M_{i}^{j}$ for all $|\alpha| \leq k+j+1, i>j,$ and for all $x \in \overline{\Omega_{j}}.$
Here $\alpha$ is a multiindex, and $D^{\alpha}$ denotes derivatives with respect to all real variables. The existence of such $\{M_{i}^{j}\}_{i > j}$ is ensured by the fact that
 $f_{i}$ is $C^{k+j+1}-$smooth for all $i>j.$

\bigskip

Next choose a sequence of positive numbers $\{a_{i}\}_{i=0}^{\infty}$ as follows:
$a_{0} < b_{0},$ and
$$a_{i}< \mathrm{min}\{b_{i},\frac{1}{2^i M_{i}^{0}},\cdots, \frac{1}{2^i M_{i}^{i-1}}\},~~ \text{for}~~
i \geq 1.$$
Let $f=\sum_{i=0}^{\infty}a_{i}f_{i}.$ Then $f$ is a $C^k$ CR function on $M.$ Moreover, from the choice of
$a_{i},i \geq 1,$ one can see that $\sum_{i=1}^{\infty} a_{i} D^{\alpha}f_{i}$ converges uniformly in $\Omega_{0},$
for any $|\alpha| \leq k+1.$ Consequently, $\sum_{i=1}^{\infty} a_{i} f_{i}$ converges to a $C^{k+1}$ function in $\Omega_{0}.$ Thus $f$ is not $C^{k+1}$ at $p_{0}$ since $f_{0}$ is not.  Similarly, one can check that $f$ is not $C^{k+i+1}$ at $p_{i},$ for all $i \geq 0.$ Hence, by the density of the sequence $\{p_{i}\}_{i=0}^{\infty}$, $f$ is a $C^k$ CR function which is not smooth on any nonempty open subset of $M$.
\end{proof}

\begin{remark}
As a consequence of Theorem \ref{thm24}, we see that for any $k \geq 1,$ there exists a CR function $f$ of class $C^{k}$  on the hypersurface $M=\{(z_{1},\cdots,z_{n},w) \in \mathbb{C}^{n+1}: \mathrm{Im}\,w=\sum_{i=1}^{n}|z_{i}|^{2}\}(n \geq 1)$  which is not smooth on any nonempty open subset, since $M$ is biholomorphically
equivalent to the unit sphere $\partial B^{n+1}=\{(z_{1},\cdots,z_{n})\in \mathbb{C}^{n+1}: |z_{1}|^{2}+\cdots+|z_{n+1}|^2=1 \}$  minus the point $(0,\dots,0,1).$
\end{remark}

We will need the following concept of $k_0-$nondegeneracy introduced in [La1]:

\begin{Definition}
Let $\widetilde{M}, \widetilde{M}'$ be CR manifolds, $\widetilde{M}'\subset \mathbb C^{N'}$ and $H:\widetilde M\to \widetilde{M}'$ a $C^{k_0}$ CR mapping, $p_{0} \in \widetilde{M}.$ Let $\rho=(\rho_{1},\cdots,\rho_{d'})$ be local defining functions for $\widetilde{M}'$ near $H(p_{0}),$ and choose a basis $L_{1},\cdots,L_{n}$ of CR vector fields for $\widetilde{M}$ near $p_{0}.$ If $\alpha=(\alpha_{1},\cdots,\alpha_{n})$ is a multiindex, write $L^{\alpha}=L_{1}^{\alpha_{1}} \cdots L_{n}^{\alpha_{n}}.$ Define the increasing sequence of subspaces $E_{i}(p_{0}) (0 \leq i \leq k_0)$ of $\mathbb{C}^{N'}$ by

$$E_{i}(p_{0})=\mathrm{Span}_{\mathbb{C}}\{L^{\alpha} \rho_{\mu,Z'}(H(Z),\overline{H(Z)})|_{Z=p_{0}}: 0 \leq |\alpha| \leq i, 1 \leq \mu \leq d'\}.$$
Here $\rho_{\mu,Z'}=(\frac{\partial \rho_{\mu}}{\partial z'_{1}},\cdots,\frac{\partial \rho_{\mu}}{\partial z'_{N'}}),$ and $Z'=(z'_{1},\cdots,z'_{N'})$ are the coordinates in $\mathbb{C}^{N'}.$ We say that $H$ is $k_{0}-$nondegenerate at $p_{0}\,\,(1 \leq k_{0} \leq k)$ if
$$E_{k_{0}-1}(p_{0}) \neq E_{k_{0}}(p_{0}) = \mathbb{C}^{N'}.$$

\end{Definition}

The dimension of $E_{i}(p)$ over $\mathbb{C}$ will be called the $i^{\mathrm{th}}$ geometric rank of $F$ at $p$ and it will be denoted by $\mathrm{rank}_{i}(F,p).$

For the invariance of the definition under the choice of the defining functions $\rho_{l},$ the basis of CR vector fields and the choice of holomorphic coordinates in $\mathbb{C}^{N'},$ the reader is referred to [La2]. It is easy to see that $\mathrm{rank}_{i}(F,p) \leq \mathrm{rank}_{i+1}(F,p),$ for any $i \geq 0, p \in M.$ We will see in Section 3 that under the hypothesis of Theorem \ref{thm1}, $\mathrm{rank}_{1}(F,p)= n+1$ and so $ \mathrm{rank}_{i}(F,p) \geq n+1,$ for any $p \in M, i \geq 1.$

\section{Proof of Theorem \ref{thm1}}
Let $M,M',F$ be as in Theorem \ref{thm1}. We work near a point $p\in M$ which we fix. If the Levi form of $M$ at $p$ has a positive and a negative eigenvalue, then the smoothness of $F$ follows from Theorem $2.9$ in [BX] and so we may assume that $M$ is strongly pseudoconvex at $p$. Let $\mathcal V$ denote the CR bundle of $M$. By Theorem IV.1.3 in [T], there is an integrable CR structure on $M$ near $p$ with CR bundle $\widehat {\mathcal V}$ that agrees with $\mathcal V$ to infinite order at $p$. In particular,  $(M, \widehat {\mathcal V})$ is strongly pseudoconvex at $p$ and hence we can find local coordinates $x_1,y_1,\dots, x_{n},y_{n}$ and $s$ vanishing at $p$ and first integrals $Z_j=x_j+\sqrt{-1}y_j=z_j,\,1\leq j\leq n$, $Z_{n+1}=s+\sqrt{-1}\psi(z,\overline{ z},s)$ where $ z=(z_1,\dots,z_{n})$ and $\psi$ is a real-valued smooth function satisfying
\[
\psi(z,\overline{ z},s)=| z|^2+O(s^2) + O(|z|^3).
\]
In these coordinates, near the origin, the bundle $\mathcal V$ has a basis of the form
$$
L_j=\frac{\partial}{\partial \overline{z_j}}+A_j( z,\overline{ z},s)\frac{\partial}{\partial s}+\sum_{k=1}^{n}B_{jk}(z,\overline{z},s)\frac{\partial}{\partial z_k}\,\,1\leq j\leq n
$$
where each
$$
A_j(z,\overline{z},s)=\frac{-\sqrt{-1}\psi_{\overline{z_j}}(z,\overline{ z},s)}{1+\sqrt{-1}\,\psi_s(z,\overline{ z},s)}\,\,\text{to infinite order at}\,0
$$
and the $B_{jk}$ vanish to infinite order at $0$. We may assume $0\in M'$, $F(0)=0$ and that we have coordinates $Z'=(z_1',\dots,z_{N+1}')$ in $\mathbb C^{N+1}$ so that near $0$, $M'$ is defined by
\begin{equation}
-\frac{z'_{N+1}-\overline{z'_{N+1}}}{2 \sqrt{-1}} + \sum_{j=1}^{l}|z'_{j}|^2  - \sum_{i=l+1}^{N} |z'_{j}|^2 +\phi^{*}(Z',\overline{Z'})=0
\end{equation}
where  $\phi^{*}(Z',\overline{Z'})=O(|Z'|^3)$ is a  real-valued smooth function.

In the following, for two $m$-tuples $x=(x_{1}.\cdots,x_{m}), y =(y_{1},\cdots,y_{m})$ of complex numbers, we write
$\langle x,y \rangle_{l}= \sum_{j=1}^m \delta_{j,l}x_{j}y_{j}$ and $|x|_{l}^2=\langle x, \overline{x} \rangle_{l}=\sum_{j=1}^m \delta_{j,l}|x_{j}|^2,$
where we denote by  $\delta_{j,l}$ the symbol which takes value $1$ when $1 \leq j \leq l$ and $-1$ otherwise. Let $ \widetilde{z'}=(z'_{1},\cdots,z'_{N}).$ Then $M'$ is locally defined by

$$\rho(Z',\overline{Z'})=-\frac{z'_{N+1}-\overline{z'}_{N+1}}{2 \sqrt{-1}}+|\widetilde{z'}|_{l}^2+ \phi^{*}(Z',\overline{Z'})=0.$$ If we write $F=(F_1,\dots, F_{N+1})=(\widetilde{F},F_{N+1}),$ then $F$ satisfies:
\begin{equation}\label{eqn1}
-\frac{F_{N+1}-\overline{F}_{N+1}}{2\sqrt{-1}}+|\widetilde{F}|_{l}^2+\phi^{*}(F,\overline{F})=0.
\end{equation}
Let $\mathcal V'$ denote the CR bundle of $M'$. Since $F$ is CR-transversal, and the fibers $\mathcal V_0$ and $\mathcal V'_0$ are spanned by $\frac{\partial}{\partial \overline{z_j}}, 1\leq j\leq n$ and $\frac{\partial}{\partial \overline{z'_k}}, 1\leq k\leq N$, we get $\lambda:= \frac{\partial F_{N+1}}{\partial s}(0) \neq 0.$ Moreover, equation $(3.2)$ shows that the imaginary part of $F_{N+1}$ vanishes to second order at the origin, and so the number $\lambda$ is real. We claim that we can assume that $\lambda > 0.$ Indeed, when $\lambda < 0,$ by considering $\widetilde{M'}$ defined by $\rho(\tau(Z),\overline{\tau(Z)})$ instead of $M',$ and considering $\widetilde{F}=\tau \circ F$ instead of $F,$ we get $\lambda >0.$ Here $\tau$ is the change of coordinates in $\mathbb{C}^{N+1}: \tau(z_{1},\cdots,z_{N},w)=(z_{1},\cdots,z_{N},-w).$
 By applying   $L_j$, $L_jL_k,\overline{L_{j}}\overline{L_{k}}$ to equation $(3.2)$, and evaluating at $0$, we get

$$\frac{\partial F_{N+1}}{\partial z_{i}}(0)=0, \,\,1 \leq i \leq n,$$
and
$$\frac{\partial F_{N+1}}{\partial z_{k} \partial z_{j}}(0)=\frac{\partial F_{N+1}}{\partial \overline{z_{k}} \partial \overline{z_{j}}}=0, \,\,\,1 \leq k ,j \leq n.$$ We next apply $\overline{L_j}L_k$ to $F_{N+1}$ and evaluate at $0$ to get
$$
\frac{\partial F_{N+1}}{\partial \overline{z_{j}} \partial {z_{k}}}(0)=\sqrt{-1}\delta_{jk}\lambda,
$$
where $\delta_{jk}$ is the Kronecker delta.
Hence we are able to write,
\begin{equation}
F_{N+1}(z,\overline{z},s)=\lambda s+\sqrt{-1}\lambda |z|^2+O(|z||s|+s^2)+o(|z|^2),
\end{equation}
For $1\leq j\leq N$, using $L_kF_j(0)=0$, we have:
\begin{equation}
F_{j}=b_{j}s+\sum_{i=1}^{n}a_{ij}z_{i} +O(|z|^2+s^2).
\end{equation}
for some $b_{j} \in \mathbb{C},a_{ij} \in \mathbb{C}, 1 \leq i \leq n, 1 \leq j \leq N,$
or equivalently,

\begin{equation}
(F_{1},\cdots,F_{N})=s(b_{1},\cdots,b_{N})+(z_{1},\cdots,z_{n})A+(\hat{F}_{1},\cdots,\hat{F}_{N})
\end{equation}
where $A=(a_{ij})_{n \times N}$ is an $n \times N$ matrix, and $\hat{F}_{j}=O(|z|^2+s^2),1 \leq j \leq N.$
Plugging $(3.3)$ and $(3.4)$ into equation $(3.2)$, we get
$$
\lambda |z|^2+O(|z||s|+s^2)+o(|z|^2)=\langle zA , \overline{z}\overline{A} \rangle_{l} +O(|z||s|+s^2)+o(|z|^2).
$$
When $s=0$ the latter equation leads to
$$
\lambda |z|^2+o(|z|^2)= \langle zA , \overline{z}\overline{A} \rangle_{l} + o(|z|^2).
$$
It follows that
\begin{equation}\label{eqn2}
\lambda I_{n}=A E(l,N)A^{*},
\end{equation}
where $A^{*}=\overline{A^{t}}$.
Here $I_{n}$ denotes the $n$ by $n$ identity matrix and $E(k,m)$ denotes the $m \times m$ diagonal matrix with its first $k$ diagonal elements $1$ and the rest $-1.$ Note from equation (\ref{eqn2}) that the matrix $A$ has rank $n.$ Moreover, since $\lambda >0,$ we get $l \geq n$ from equation (\ref{eqn2}) using elementary linear algebra. Since $l \leq n$, it follows that $l=n$. Thus $M'$ is locally defined by
\begin{equation}\label{eqn04}
\rho(Z',\overline{Z'})=-\frac{z'_{N+1}-\overline{z'}_{N+1}}{2 \sqrt{-1}}+|\widetilde{z'}|_{n}^2+ \phi^{*}(Z',\overline{Z'})=0,
\end{equation}
and we have
\begin{equation}\label{eqn041}
\lambda I_{n}=A E(n,N)A^{*}.
\end{equation}

A direct computation shows that
$$L_{i}\overline{F}_{j}(0)=\overline{a}_{ij},\,\, 1 \leq i \leq n, 1 \leq j \leq N.$$
Since $A=(a_{ij})_{1 \leq i \leq n, 1 \leq j \leq N}$
is of rank $n,$ we conclude that $dF: T_{0}^{(0,1)}M \rightarrow T_{0}^{(0,1)}M'$ is injective.
Now let us introduce some notations. Set
$$a_{j}(Z,\overline{Z})=\rho_{z'_{j}}(F(Z),\overline{F(Z)})=\delta_{j,n}\overline{F}_{j}+\phi_{z'_{j}}^{*}(F,\overline{F}),~~1 \leq j \leq N,$$
$$a_{N+1}(Z,\overline{Z})=\rho_{z'_{N+1}}(F(Z),\overline{F(Z)})=\frac{\sqrt{-1}}{2}+\phi_{z'_{N+1}}^{*}(F,\overline{F}).$$
and $${\bf a}=(a_{1},\cdots,a_{N+1}).$$
We have:
$$L^{\alpha}\rho_{Z'}(F,\overline{F})=L^{\alpha}{\bf{a}}=(L^{\alpha}a_{1},\cdots,L^{\alpha}a_{N}
,L^{\alpha}a_{N+1}),$$
for any multiindex $ 0 \leq  |\alpha| \leq N-n+1.$ Recall that for any $0 \leq i \leq N-n+1,$

$$\mathrm{rank}_{i}(F,p)=\mathrm{dim}_{\mathbb{C}}(\mathrm{Span}_{\mathbb{C}}\{L^{\alpha}{\bf{a}}(Z,\overline{Z})|_{p}:0 \leq |\alpha| \leq i\}).$$

From the injectivity of $dF$ and we get
\begin{lemma}
Let $M,M',F$ be as in Theorem \ref{thm1}. Then for any  $ p \in M, \mathrm{rank}_{0}(F,p)=1, \mathrm{rank}_{1}(F,p)=n+1.$ Consequently, $\mathrm{rank}_{i}(F,p) \geq n+1,$ for any $i \geq 1.$
\end{lemma}

We next prove a normalization lemma which will be used later.
\begin{lemma} \label{lemma1}
Let $M, M', F$ be as in Theorem \ref{thm1}. Assume $rank_{l}(F,p)=m+1,$ for some $l>1, m \geq n.$ Then there exist multiindices $\{ \beta_{n+1},\cdots,\beta_{m}\}$ with $1 < |\beta_{i}| \leq l$ for all $i,$ such that after a linear biholomorphic change of coordinates in $\mathbb{C}^{N+1}: \widetilde{Z}=(\widetilde{z}_{1},\cdots,\widetilde{z}_{N},\widetilde{z}_{N+1})
=((z'_{1},\cdots,z'_{N})V, z'_{N+1}),$ where $\widetilde{Z}$ denotes the new coordinates in $\mathbb{C}^{N+1},$ and $V$ is an $N \times N$ matrix satisfying $V E(n,N) V^{*}=E(n,N),$ the following hold:
\begin{equation} \label{eqn05}
\widetilde{\bf{a}}|_{p}=(0,\cdots,0,\frac{\sqrt{-1}}{2}),
\left(\begin{array}{c}
  L_{1}\widetilde{\bf{a}}|_{p} \\
  \cdots \\
  L_{n}\widetilde{\bf{a}}|_{p}
  \end{array}\right)=\left(\begin{array}{ccc}
                     \sqrt{\lambda} {\bf{I}}_{n} & {\bf{0}}_{n \times (N-n)} & {{\bf{0}}_{n}^t}
                   \end{array}\right),
\end{equation}

\begin{equation}\label{eqn06}
\left(\begin{array}{c}
  L^{\beta_{n+1}}\widetilde{\bf{a}}|_{p} \\
  \cdots \\
  L^{\beta_{m}}\widetilde{\bf{a}}|_{p}
  \end{array}\right)=\left(\begin{array}{cccc}
                             {\bf C} & {\bf M}_{m-n} & {\bf 0}_{(m-n) \times (N-m)} & {\bf d}
                           \end{array}
  \right).
\end{equation}

Here we write $\widetilde{\bf{a}}=\widetilde{\rho}_{\widetilde{Z}}(\widetilde{Z}(F),\overline{\widetilde{Z}(F)}),$ and $\widetilde{\rho}$ is a local defining function of $M'$ near $0$ in the new coordinates. Moreover,  ${\bf{I}}_{n}$ is the $n \times n$ identity matrix, ${\bf{0}}_{n \times (N-n)}$ is an $n \times(N-n)$ zero matrix, and ${\bf{0}}_{n}^t$ is an $n-$dimensional zero column vector. ${\bf C}$ is an $(m-n) \times n$ matrix, ${\bf M}_{m-n}$ is an $(m-n) \times (m-n)$ invertible matrix, ${\bf 0}_{(m-n) \times (N-m)}$ is an $(m-n) \times (N-m)$ zero matrix, and ${\bf d}$ is an $(m-n)-$dimensional column vector.
\end{lemma}
\begin{proof}
Assume that $p=0.$ Note that
$L_{i}a_{j}(0)=\delta_{j,n} L_{i}\overline{F}_{j}(0)=\delta_{j,n} \overline{a}_{ij}.$ Thus we have,
$$\left(\begin{array}{c}
    {\bf a}|_{0} \\
    L_{1}{\bf a}|_{0} \\
    \cdots \\
    L_{n}{\bf a}|_{0}
  \end{array}\right)=\left(\begin{array}{cc}
                       {\bf 0}_{N-n} & \frac{\sqrt{-1}}{2} \\
                       \overline{A}E(n,N) & {{\bf 0}_{n}^t}
                     \end{array}\right),
$$
where $A=(a_{ij})_{1 \leq i \leq n, 1 \leq j \leq N}$ is as mentioned above, ${\bf 0}_{N-n}$ is an $(N-n)-$ dimensional zero row vector, ${\bf 0}_{n}^t$ is an $n-$dimensional zero column vector. Let $B=E(n,N) A^{t}.$ Then
by equation (\ref{eqn041}), we know that $\overline{A}B=\lambda I_{n},$ and $B^{*}E(n,N)B=\lambda I_{n}.$ By a result in [BHu] (see page 386 in [BHu] for more details on this), we can find an $N \times N$ matrix
$U$ whose first $n$ rows are  rows of $B^{*},$ such that, $U E(n,N) U^{*}=\lambda E(n,N).$ Consequently, $U^{*}E(n,N)U=\lambda E(n,N), W E(n,N)W^{*}=E(n,N),$ where $W=\frac{1}{\sqrt{\lambda}} U^{*}.$

We next make the following change of coordinates in $\mathbb{C}^{N+1}$:  $\widetilde{Z}=Z' D^{-1}$ where
$$D=\left(\begin{array}{cc}
            E(n,N)W & {\bf 0}_{N}^t \\
            {\bf 0}_{N} & 1
          \end{array}\right),
$$
and $0_{N}$ is $N-$dimensional zero row vector.
Then the function $\widetilde{\rho}(\widetilde{Z},\overline{\widetilde{Z}})=\rho(\widetilde{Z}D,\overline{\widetilde{Z}D})$
is a defining function for $M'$ near $0$ with respect to the new coordinates $\widetilde{Z}.$ By the chain rule,
$$\widetilde{\rho}_{\widetilde{Z}}(\widetilde{F}(Z),\overline{\widetilde{F}(Z)})=
\rho_{Z'}(F(Z),\overline{F(Z)})D,\quad \text{where}\,\, \widetilde{F}(Z)=F(Z)D^{-1}.$$
For any multiindex $\alpha,$
 $$L^{\alpha} \widetilde{\rho}_{\widetilde{Z}}(\widetilde{F}(Z),\overline{\widetilde{F}(Z)})=
L^{\alpha} \rho_{Z'}(F(Z),\overline{F(Z)})D.$$
In particular, at $p=0,$ we have,
$$\left(\begin{array}{c}
    {\bf \widetilde{a}}|_{0} \\
    L_{1}{\bf \widetilde{a}}|_{0} \\
    \cdots \\
    L_{n}{\bf \widetilde{a}}|_{0}
  \end{array}\right)=\left(\begin{array}{cc}
                       {\bf 0}_{N} & \frac{\sqrt{-1}}{2} \\
                       \overline{A}E(n,N) & {{\bf 0}_{n}^t}
                     \end{array}\right)D=\left(\begin{array}{cc}
                                           {\bf 0}_{N} & \frac{\sqrt{-1}}{2} \\
                                           \overline{A}W & {{\bf 0}_{n}^t}
                                         \end{array}\right), $$

where $\widetilde{\bf a}(Z,\overline Z)= \widetilde{\rho}_{\widetilde{Z}}(\widetilde{F}(Z),\overline{\widetilde{F}(Z)})$.  Since  $\overline{A}=B^{*}E(n,N)$,
$$\overline{A}W=\frac{1}{\sqrt{\lambda}}B^{*}E(n,N)U^{*}=\left(\begin{array}{cc}
                                       \sqrt{\lambda} I_{n} & {\bf 0}
                                     \end{array}\right).
$$

Thus equation (\ref{eqn05}) holds with respect to the new coordinates $\widetilde{Z}.$ In the following, we will still write $Z'$ instead of $\widetilde{Z}, {\bf a}$ instead of $\widetilde{\bf a}.$
Since $\{{\bf a},L_{1}{\bf a}, \cdots,L_{n}{\bf a}\}|_{0}$ is linearly independent, extend it to a basis of $E_{l}(0),$ which has dimension $m+1$ by assumption. That is, pick multiindices $\{ \beta_{n+1},\cdots,\beta_{m}\}$ with $1 < |\beta_{i}| \leq l$ for each $i,$ such that,
$$\{{\bf a}, L_{1}{\bf a}, \cdots, L_{n}{\bf a},L^{\beta_{n+1}}{\bf a},\cdots,L^{\beta_{m}}{\bf a}\}|_{0}$$
is linearly independent over $\mathbb{C}.$ Write $\hat{\bf a}=(a_{n+1},\cdots,a_{N}),$ i.e., the $(n+1)^{\mathrm{th}}$ to $N^{\mathrm{th}}$ components of ${\bf a}.$
Note that $\{{\bf a},L_{1}{\bf a},\cdots,L_{n}{\bf a}\}|_{0}$ is of the form (\ref{eqn05}). The set $\{L^{\beta_{n+1}}\hat{ \bf a}, \cdots, L^{\beta_{m}}\hat{\bf a}\}|_{0}$
is linearly independent in $\mathbb{C}^{N-n}.$ Let $S$ be the $(m-n)-$dimensional vector
space spanned by it and let $\{T_{1},\cdots,T_{m-n}\}$ be an orthonormal basis of $S.$
Extend it to an orthonormal basis $\{T_{1},\cdots,T_{m-n},T_{m-n+1},\cdots,T_{N-n}\}$ of $\mathbb{C}^{N-n}$ and set $T$ to be the following  $(N-n) \times (N-n)$ unitary matrix:
$$T=\left(\begin{array}{c}
      T_{1} \\
      \cdots \\
      T_{N-n}
    \end{array}\right)^{*}.
$$
We next make the following change of coordinates:
 $$\widetilde{Z}=(\widetilde{z}_{1}, \cdots, \widetilde{z}_{N},\widetilde{z}_{N+1})=(z'_{1},\cdots,z'_{n},(z'_{n+1},\cdots,z'_{N})T^{-1}, z'_{N+1}).$$
One can check that equation (\ref{eqn06}) holds in the new coordinates $\widetilde{Z}$.
\end{proof}

\begin{remark}
From the construction of $~V$ in the proof of Lemma \ref{lemma1}, one
can see that, in the new coordinates $\widetilde{Z}$, the following continues to hold: $M'$ is locally defined near $0$ by
$$\widetilde{\rho}(\widetilde{Z},\overline{\widetilde{Z}})=
-\frac{\widetilde{z}_{N+1}-\overline{\widetilde{z}_{N+1}}}{2\sqrt{-1}}+\sum_{i=1}^{n}|\widetilde{z}_{i}|^2
-\sum_{i=n+1}^{N}|\widetilde{z}_{i}|^2
+\widetilde{\phi^{*}}(\widetilde{Z},\overline{\widetilde{Z}})=0,$$
where $\widetilde{Z}=(\widetilde{z}_{1},\cdots,\widetilde{z}_{N},\widetilde{z}_{N+1}),$
and $\widetilde{\phi^{*}}(\widetilde{Z}, \overline{\widetilde{Z}})=O(|\widetilde{Z}|^3)$ is a real-valued smooth function near $0.$ In what follows, we will write the new coordinates as $Z'$ instead of $\widetilde{Z}$,  drop the tilde from $\widetilde {\rho}$ and set ${\bf{a}}(Z,\overline Z)=\rho_{Z'}(F(Z),\overline{F(Z)})$.
\end{remark}

\begin{remark}
In Lemma \ref{lemma1}, equations (\ref{eqn05}), (\ref{eqn06}) can be rewritten as follows:
\begin{equation} \label{eqnnorm}
\widetilde{\bf{a}}|_{0}=(0,\cdots,0,\frac{\sqrt{-1}}{2}),
\left(\begin{array}{c}
  L_{1}\bf{a}|_{0} \\
  \cdots \\
  L_{n}\bf{a}|_{0} \\
  L^{\beta_{n+1}}\bf{a}|_{0}  \\
  \cdots \\
  L^{\beta_{m}}\bf{a}|_{0}
\end{array}\right)=\left(\begin{array}{ccc}
                     {\bf{B}}_{m} & {\bf{0}} & {\bf{b}}
                   \end{array}\right),
\end{equation}
where ${\bf B}_{m}$ is an $m \times m$ invertible matrix, ${\bf 0}$ is an $m \times (N-m)$ zero matrix, ${\bf b}$ is an $m-$dimensional column vector.
 We note that Lemma \ref{lemma1} plays the same role as Lemma $4.2$ in [BX].
\end{remark}

The remaining argument will be essentially the same as in [BX]. But to make the paper more complete and self-contained, we  will still include a few details. First we need the following regularity theorem from [BX](Theorem 4.8, [BX]).
\begin{theorem}\label{thm4}
Let $M,M',F$ be as in Theorem \ref{thm1} (resp. as in Theorem \ref{thm2.1}). Let $p \in M$ and $O$ be a
neighborhood of $p$ in $M.$ Assume that for some $1
\leq l \leq N-n,$
$\mathrm{rank}_{l}(F,p)=n+l$, and $\mathrm{rank}_{l+1}(F,q) = n+l
~\text{for all}~q \in O.$  Then F is smooth (resp. real analytic) near $p.$
\end{theorem}
\begin{proof}
We first prove Theorem \ref{thm4} in the smooth case. Although $M'$ is different from the one in [BX],  the proof of theorem 4.8 in [BX] applies to establish Theorem \ref{thm4} which involves applications of Lemma \ref{lemma1} above and Theorem
V.3.7 in [BCH]. Assume $p=0.$ From Lemma 3.2 and the assumption, after a suitable biholomorphic change of coordinates, we conclude that there exist multiindices $\{\beta_{n+1},\dots, \beta_{n+l-1}\}$ with $1 < |\beta_{i}| \leq l,$ such that
\begin{equation}
\widetilde{\bf{a}}|_{0}=(0,\cdots,0,\frac{\sqrt{-1}}{2}),
\left(\begin{array}{c}
  L_{1}\bf{a}|_{0} \\
  \cdots \\
  L_{n}\bf{a}|_{0} \\
  L^{\beta_{n+1}}\bf{a}|_{0}  \\
  \cdots \\
  L^{\beta_{n+l-1}}\bf{a}|_{0}
\end{array}\right)=\left(\begin{array}{ccc}
                     {\bf{B}}_{n+l-1} & {\bf{0}} & {\bf{b}}
                   \end{array}\right).
\end{equation}

 Indeed, the form (3.12) is all that is needed to use the proof of Theorem 4.8 to arrive at the following:

There are CR functions $G_i^j$ of smoothness class $C^{N+1-n-l}$ defined in a neighborhood $O$ of $0$ in $M$ such that :
\begin{equation}
a_{j}-\sum_{i=1}^{n+l-1}G_{i}^{j}a_{i}-G_{N+1}^{j}a_{N+1}=0, \,\, n+l \leq j \leq N.
\end{equation}
That is, in $O$,
\begin{equation}\label{eqneG}
\delta_{j,n}F_{j}+\overline{\phi_{z'_{j}}^{*}}-\sum_{i=1}^{n+l-1}\overline{G_{i}^{j}}(\delta_{i,n}F_{i}+\overline{\phi_{z'_{i}}^{*}})-
\overline{G_{N+1}^{j}}(\frac{1}{2\sqrt{-1}}+\overline{\phi_{z'_{N+1}}^{*}})=0.
\end{equation}
We also have,

\begin{equation}\label{eqFde}
-\frac{F_{N+1}-\overline{F_{N+1}}}{2\sqrt{-1}}+F_{1}\overline{F_{1}}+
\cdots +F_{n}\overline{F_{n}}-F_{n+1}\overline{F_{n+1}}- \cdots - F_N\overline{F_N}+\phi^{*}(F,\overline{F})=0;
\end{equation}
for $1 \leq j \leq n,$
\begin{equation}\label{eqlFg}
\frac{L_{j}\overline{F_{N+1}}}{2\sqrt{-1}}+F_{1}L_{j}\overline{F_{1}}+\cdots+F_{n}
L_{j}\overline{F_{n}} - F_{n+1}L_j\overline{F_{n+1}} -\cdots - F_N L_j\overline{F_N}+L_{j}\phi^{*}(F,\overline{F})=0;
\end{equation}
and for $n+1 \leq t \leq n+l-1,$
\begin{equation}\label{eqlFd}
\frac{L^{\beta_{t}}\overline{F_{N+1}}}{2\sqrt{-1}}+F_{1}L^{\beta_{t}}\overline{F_{1}}+\cdots+F_{n}
L^{\beta_{t}}\overline{F_{n}}- F_{n+1}L^{\beta_{t}}\overline{F_{n+1}}- \cdots - F_{N}L^{\beta_{t}}\overline{F_{N}}+L^{\beta_{t}}\phi^{*}(F,\overline{F})=0.
\end{equation}

We recall the local coordinates $(x,y,s)\in \mathbb R^n\times \mathbb R^n\times \mathbb R$ that vanish at the central ponit $p\in M$. By Theorem 2.9 in [BX],  $G_{i}^{j},G_{N+1}^{j},F_{1},\cdots,F_{N+1}$ extend to almost analytic functions into a half-space $\{(x,y,s+it)\in U\times V\times \Gamma: \,(x,y,s)\in U\times V,\,t\in \Gamma\}$, with edge $M$ near $p=0$ for all $1 \leq i \leq n+l-1, n+l \leq j \leq N.$ Here $U\times V$ is a neighborhood of the origin in $\mathbb{C}^n \times \mathbb R$ and $\Gamma$ is an interval $(0,r)$ in $t-$space. We still denote the extended functions by $G_{i}^{j},G_{N+1}^{j},F_{1},\cdots,F_{N+1}.$  \bigskip

Equations $(3.14), (3.15), (3.16)$ and $(3.17)$ can be used to get a smooth map \newline $\Psi(Z',\overline{Z'},W)=(\Psi_{1},\cdots,\Psi_{N+1})$ defined in a neighborhood of $\{0\} \times \mathbb{C}^{q}$ in $\mathbb{C}^{N+1} \times \mathbb{C}^{q},$~ smooth in the first $N+1$ variables and polynomial in the last $q$ variables for some integer $q$, such that,
$$\Psi(F,\overline{F},(L^{\alpha}\overline{F})_{1 \leq |\alpha| \leq l},
\overline{G_{1}^{n+l}},\cdots,\overline{G^{n+l}_{n+l-1}},\overline{G_{N+1}^{n+l}},\cdots,\overline{G_{1}^{N}},\cdots,
\overline{G_{n+l-1}^{N}},\overline{G_{N+1}^{N}})=0$$
at $(z,s,0)$ with $(z,s) \in U \times V.$ Write
\begin{equation}\label{eqtng}
\overline{G}=(\overline{G_{1}^{n+l}},\cdots,\overline{G^{n+l}_{n+l-1}},\overline{G_{N+1}^{n+l}},\cdots,\overline{G_{1}^{N}},\cdots,
\overline{G_{n+l-1}^{N}},\overline{G_{N+1}^{N}}).
\end{equation}
Observe that
$$\Psi_{Z'}|_{(F(0),\overline{F}(0),(L^{\alpha}\overline{F})_{1\leq |\alpha|\leq l}(0),\overline{G}(0))}=\left(\begin{array}{ccc}
                   {\bf{0}}_{n+l-1} & {\bf{0}}_{N-n-l+1} & \frac{\sqrt{-1}}{2} \\
                   {\bf{B}}_{n+l-1} & {\bf{0}} & {\bf{b}} \\
                   {\bf{C}} & -{\bf{I}}_{N-n-l+1} & {\bf{0}}^{t}_{N-n-l+1}
                 \end{array}\right),
$$
where ${\bf{0}}_{m}$ is an $m-$dimensional zero row vector,
${\bf{C}}$ is a $(N-n-l+1)\times (n+l-1)$ matrix, ${\bf{I}}_{N-n-l+1}$
is the  $(N-n-l+1) \times (N-n-l+1)$ identity matrix and we recall that ${\bf{B}}_{n+l-1}$ is an invertible $(n+l-1)\times
(n+l-1)$ matrix, ${\bf{0}}$ is an $(n+l-1) \times(N-n-l+1)$ zero
matrix, ${\bf{b}}$ is an $(n+l-1)-$dimensional column vector.

 The matrix $\Psi_{Z'}|_{(F(0),\overline{F(0)},(L^{\alpha}\overline{F})_{1 \leq |\alpha|\leq l}(0),\overline{G}(0))}$ is invertible. By applying the ``almost holomorphic" implicit function theorem in [La1], we get a solution $\psi=(\psi_{1},\cdots,\psi_{N+1})$  from $\mathbb{C}^{N+1} \times \mathbb{C}^{q}$ to $\mathbb{C}^{N+1}$ satisfying for each multiindex $\alpha$, and each $j$,
 $$ D^{\alpha}\frac{\partial \psi_j}{\partial Z'_i}(Z',\overline{Z'},W)=0,\,\,\,\text{if}\,\,Z'=\psi(Z',\overline{Z'},W)$$
 and for each $1 \leq j \leq N+1,$

$$F_{j}=\psi_{j}(F,\overline{F},(L^{\alpha}\overline{F})_{1 \leq |\alpha| \leq l},
\overline{G})$$
at $(z,s,0)$ with $(z,s) \in U \times V.$
The map $\psi$ is smooth in all variables and holomorphic in $W$. For each $j=1,\cdots,n,$ we denote by $M_{j}$ smooth extensions of $L_{j}$ to $U \times V \times \mathbb{R}$ given by

$$
M_{j}=\frac{\partial}{\partial \overline{z}_{j}} + A(x,y,s,t)\frac{\partial}{\partial s} + \sum_{k=1}^{n} B_{jk}(x,y,s,t)\frac{\partial}{\partial z_{k}}
$$
where the $B_{jk}$ and $A$ are smooth extensions of the corresponding coefficients of the $L_j$ satisfying
\begin{equation}\label{eqn220}
\overline{\partial}_wA(x,y,s,t),\overline{\partial}_wB_{jk}(x,y,s,t)=O(|t|^{m}),\,\,\,\forall m=1,2,\cdots.
\end{equation}

For each $1 \leq j \leq N+1,$ set
$$h_{j}(z,s,t)=\psi_{j}(F(z,s,-t),\overline{F}(z,s,-t),(M^{\alpha}\overline{F})_{1 \leq |\alpha| \leq l}(z,s,-t),\overline{G}(z,s,-t))$$ and shrink $U$ and $V$ and
choose $\delta$ in such a way that each $h_{j}$  is defined and
continuous in $\overline{\Omega_{-}}$ where $\Omega_{-}=\{(x,y,s+it):\,(x,y,s)\in U\times V,\,t\in -\Gamma, |t| \leq \delta \}$. The arguments in [BX] showed the estimates:

\[
\left |D_x^{\alpha}D_y^{\beta}D_s^{\gamma}h_j(z,s,t)\right |\leq \frac{C}{|t|^{\lambda}},\,\,\text{for some $C,\lambda>0$}
\]
and
\[
 D_x^{\alpha}D_y^{\beta}D_s^{\gamma}\overline{\partial}_{w}h_j(z,s,t)=O(|t|^m),\,\, \forall  m=1,2,\dots
 \] for $t\in -\Gamma$, $1 \leq j \leq N+1.$

Notice that the $F_{j}$ satisfy similar estimates for $t\in\Gamma$, and  $b_{+}F_{j}=b_{-}h_{j}$ for each $1\leq j \leq N+1.$ Applying Theorem V.3.7 in [BCH], we conclude that $F$ is smooth near $p$. This establishes Theorem $3.5$ in the smooth case.

\bigskip

The proof of Theorem \ref{thm4} in the real analytic case is similar and so we will only briefly indicate the modifications that are needed. With $M, M', F$ as in Theorem \ref{thm2.1}, we will show that the map $F$ is real analytic at $p$ which we assume is the origin. Since $\phi^{\ast}$ and the $L_j$ are real analytic now, equations $(3.14)-(3.17)$ imply that there is a real analytic map  \newline $\Psi(Z',\overline{Z'},W)=(\Psi_{1},\cdots,\Psi_{N+1})$ defined in a neighborhood of $\{0\} \times \mathbb{C}^{q}$ in $\mathbb{C}^{N+1} \times \mathbb{C}^{q}$, polynomial in the last $q$ variables for some integer $q$, such that,
$$\Psi(F,\overline{F},(L^{\alpha}\overline{F})_{1 \leq |\alpha| \leq l},
\overline{G_{1}^{n+l}},\cdots,\overline{G^{n+l}_{n+l-1}},\overline{G_{N+1}^{n+l}},\cdots,\overline{G_{1}^{N}},\cdots,
\overline{G_{n+l-1}^{N}},\overline{G_{N+1}^{N}})=0$$
at $(z,s,0)$ with $(z,s) \in U \times V$. Since the matrix $\Psi_{Z}$ is invertible at the central point, by the holomorphic version of the implicit function theorem, we get a holomorphic map $\psi=(\psi_1,\dots,\psi_{N+1})$ such that near the origin,

$$
F_ j=\psi_j(\overline F,(L^{\alpha}\overline{F})_{1\leq |\alpha|\leq l},\overline G),\,\,\,1\leq j\leq N+1,
$$
where $\overline{G}$ is as in equation  $(\ref{eqtng}).$ We may assume that near the origin, $M$ is given by $\{(z,w)\in \mathbb C^n\times \mathbb C:\,\mathrm{Im}\,w=\varphi(z,\overline z,s)\}$, where $\varphi$ is a real-valued, real analytic function with $\varphi(0)=0$, and $d\varphi(0)=0$. In the local coordinates $(z,s)\in \mathbb C^n\times \mathbb R$, we may assume that
$$
L_j=\frac{\partial}{\partial \overline{z}_j}-i\frac{\varphi_{\overline{z}_j}(z,\overline z,s)}{1+i\varphi_s(z,\overline z,s)}\frac{\partial}{\partial s},\,\,\,1\leq j\leq n.
$$
Since $\varphi$ is real analytic, we can complexify in the $s$ variable and write

$$
M_j=\frac{\partial}{\partial \overline{z}_j}-i\frac{\varphi_{\overline{z}_j}(z,\overline z,s+it)}{1+i\varphi_s(z,\overline z,s+it)}\frac{\partial}{\partial s},\,\,\,1\leq j\leq n
$$
which are holomorphic in $s+it$ and extend the vector fields $L_j$.
For each $1 \leq j \leq N+1,$ set
$$h_{j}(z,s,t)=\psi_{j}(\overline{F}(z,s,-t),(M^{\alpha}\overline{F})_{1 \leq |\alpha| \leq l}(z,s,-t),\overline{G}(z,s,-t)).$$
Since $M$ is strongly pseudo convex, the CR functions $F_j$ and $G_i$ all extend as holomorphic functions in $s+it$ to the side $t>0$. Hence the conjugates $\overline F_j(z,\overline z,s,-t)$ and $\overline G_i(z,\overline z,s,-t)$ extend holomorphically to the side $t<0$. It now follows that the $F_j$ extend as holomorphic functions to a full neighborhood of the origin (see Lemma 9.2.9 in [BER]).  This establishes Theorem $3.5$ in the real analytic case.

\end{proof}

\bigskip
{\bf End of the proof of Theorem \ref{thm1}:} Let $$\Omega_{1}=\{ p \in M : \mathrm{rank}_{N-n+1}(F,p)=N+1\},$$

$$\Omega_{2}=\{ p \in M: \mathrm{rank}_{N-n+1}(F,q) \leq N~\text{for all}~q~\text{in a neighborhood of}~p \},$$

$$\Omega=\{ p \in M : F~\text{is smooth in a neighborhood of}~p\} $$

Let $p\in \Omega_1$. Since $\mathrm{rank}_1(F,p)=n+1<N+1$, there is a minimum $m,\, 1<m\leq N-n+1$ such that $\mathrm{rank}_m(F,p)=N+1$.  By Theorem 2.3 in [BX], it follows that $F$ is smooth near $p,$ for any $p \in \Omega_{1},$ i.e., $ \Omega_{1} \subset \Omega.$
If $ p \in \Omega_{2}$ there is a neighborhood $\widetilde{O}$ of $p$, an integer $ 2 \leq d \leq N-n+1,$ and a sequence $\{p_{i}\}_{i=0}^{\infty} \subset \widetilde{O}$ converging to $p$ such that the following hold: $\mathrm{rank}_{d}(F,q)\leq n+d-1$ for all $q \in \widetilde{O},$ and $\mathrm{rank}_{d-1}(F,p_{i})=n+d-1,$ for all $i \geq 0.$ By applying Theorem $3.5$, $F$ is smooth near each $p_{i}$. Thus $\Omega$ is dense in $\Omega_{1} \cup \Omega_{2} $ and therefore dense in $M.$ This establishes Theorem \ref{thm1}.

\bigskip

{\bf Proof of Theorem $2.2$:} Let $\Omega_{1}, \Omega_{2}$ be as in the proof of Theorem $2.1.$ Note that at a point $p\in \Omega_1$, that is, at a point where the map $F$ is non-degenerate, Theorem 2 of [La2] shows that $F$ is real analytic. Thus as in the proof of Theorem $2.1,$ by applying Theorem \ref{thm4} in the real analytic case, we establish that $F$ is real analytic on a dense open subset of $M$.

\bibliographystyle{amsalpha}

\end{document}